\documentclass[11pt, twoside]{amsart}

\usepackage{amssymb, amsmath, amsthm, mathrsfs,}
\usepackage{graphicx, epsfig, color,wrapfig}
\usepackage{tikz}
\usepackage{hyperref}
\hypersetup{
    colorlinks,%
    citecolor=black,%
    filecolor=black,%
    linkcolor=black,%
    urlcolor=blue
}

\usepackage{todonotes}

\usepackage{lipsum}
\usepackage{floatrow}
\usepackage[latin1]{inputenc}
\usepackage[all,cmtip]{xy}

\usepackage{chngpage}

\usepackage{booktabs}

\newcommand*{\DashedArrow}[1][]{\mathbin{\tikz [baseline=-0.25ex,-latex, dashed,#1] \draw [#1] (0pt,0.5ex) -- (1.3em,0.5ex);}}%

\theoremstyle{plain}
\newtheorem{teo}{Theorem}[section]

\newtheorem{lema}[teo]{Lemma}
\newtheorem{claim}[teo]{Claim}
\theoremstyle{remark}
\newtheorem{rem}{Remark}
\newtheorem{ej}[teo]{Example}
\theoremstyle{definition}
\newtheorem{mydef}[teo]{Definition}

\newcommand{\exedout}{
\begin{tikzpicture}[scale=0.8]
(8,0) rectangle +(5,5);
\draw[help lines] (8,0) grid +(5,5);
\draw[dashed] (8,0) -- +(5,5);
\coordinate (prev) at (8,0);
\draw [color=blue!60, line width=2] (8,0)--(9,0)--(10,0)--(10,1)--(13,1)--(13,3)--(13,5);
\draw [color=red!60, line width=1] (8,0)--(9,0)--(10,0)--(10,1)--(10,2)--(13,2)--(13,3)--(13,5);

\draw (10,1) node [scale=0.5, circle, draw,fill=blue!60]{};
\draw (13,2) node [scale=0.5, circle, draw,fill=red!60]{};
 \end{tikzpicture}
}

\newcommand{\exedouttwo}{
\begin{tikzpicture}[scale=0.8]
(0,0) rectangle +(5,5);
\draw[help lines] (0,0) grid +(5,5);
\draw[dashed] (0,0) -- +(5,5);
\coordinate (prev) at (0,0);

\draw [color=blue!60, line width=2.5] (0,0)--(1,0)--(1,1)--(3,1)--(3,3)--(4,3)--(5,3)--(5,5);
\draw [color=red!60, line width=1.5] (0,0)--(1,0)--(1,1)--(4,1)--(4,3)--(5,3)--(5,5);

\draw (3,1) node [scale=0.5, circle, draw,fill=blue!60]{};
\draw (4,3) node [scale=0.5, circle, draw,fill=red!60]{};

\end{tikzpicture}
}

\title[Invariants of the {B}rill-{N}oether curve]
{Invariants of the {B}rill-{N}oether curve}

\author[A. Castorena]{Abel Castorena}
\address{Centro de Ciencias Matem\'aticas, Unidad Morelia, Universidad Nacional Aut\'onoma de M\'exico, Apartado Postal 61-3 (Xangari), 58089 Morelia, Michoac\'an}
\email{abel@matmor.unam.mx}

\author[A. L\'opez]{Alberto L\'opez Mart\'in}
\author[M. Teixidor]{Montserrat Teixidor {\textrm i} Bigas} \address{Department of Mathematics, Tufts University, Bromfield-
Pearson Hall, 503 Boston Avenue, Medford, MA 02155}
 \email{alberto.lopez@tufts.edu}
 \email{montserrat.teixidoribigas@tufts.edu}


\thanks{The first author was supported by grants IN100211-2 (PAPIIT--UNAM) and 166158 (CONACyT)}
\date{\today}
\begin{document}

\begin{abstract}
For a projective nonsingular  curve of genus $g$, the Brill-Noether locus $W^r_d(C)$ parametrizes line 
bundles of degree $d$ over $C$ with at least $r+1$ sections. When the curve is generic and the Brill-Noether number $\rho(g,r,d)$ equals $1$, one can then talk of the \emph{Brill-Noether curve}. In this paper, we explore the first two invariants of this curve, giving a new way of calculating the genus of this curve and computing its gonality when $C$ has genus 5.
\end{abstract}

\maketitle
\begin{section}*{Introduction}
Let $C$ be a projective nonsingular  curve of genus $g$ defined over an algebraically closed field
and $\mbox{Pic}^d(C)$  the Picard variety that parametrizes isomorphism classes of degree $d$ 
line bundles on $C$. Let us denote by  $W^r_d(C)$  the subvariety of $\mbox{Pic}^d(C)$ that parametrizes line 
bundles  of degree $d$ over $C$ with at least $r+1$ linearly independent sections. The expected dimension of $W^r_d(C)$ is given 
by the Brill-Noether number $$\rho(g,r,d)=g-(r+1)(g-d+r).$$ 
Over an open set in the moduli space of curves, the expected dimension of $W^r_d(C)$ is actually its dimension.
 In particular, when $\rho(g,r,d)=1$ and $C$ is generic, the Brill-Noether locus $W^r_d(C)$ is a curve we will call the \emph{Brill-Noether curve}. One can then define the rational map 
$$\phi:\overline{M}_g\DashedArrow[->,densely dashed] \overline{M}_{g'}$$  by $\phi([C])=[W_d^r(C)]$, where $g'$ is the genus $g_{W^r_d}$ of the Brill-Noether curve. Pirola \cite{Pir85} and
  Eisenbud-Harris \cite{EH87}, using the determinantal adjunction formula of Harris-Tu 
\cite{HT84}, calculated this genus as 
 \begin{equation}g_{W^r_d}=1+\frac{g-d+r}{g-d+2r+1}\prod_{i=0}^r\frac{i!}{(g-d+r+i)!}\cdot g!\end{equation}

While this gives a collection of rational maps between moduli spaces of curves, their images are mostly unknown. 
Farkas \cite{Fa10} and Ortega \cite{Ort13} have recently worked on several questions around or using this map and posed questions about the
 gonality of the image curve under the map $\phi$. 

In this note, we provide a new method for computing the genus of the Brill-Noether curve when $r=1$. We also find its gonality in the first interesting case,
namely when $g=5$ and $r=1$. 
Our techniques provide  a different approach to the computation of  the Castelnuovo numbers,
 that is, the number of linear series of fixed degree and dimension when this number is finite (see Theorem \ref{castelnuovo}).

Our results are obtained by degenerating the given curve to a chain of elliptic components and describing explicitly the Brill-Noether curve in this situation. 
Our method could in principle, be extended to values $r>1$ and, in the case of the gonality, to values of $g$ other than 5.
 In both situations however, the combinatorics are more involved.
 
It is worth mentioning that the combinatorics arising in the description of the limit linear series over the class of curves we consider is very similar to the one that appears in the tropical proof of the Brill-Noether theorem in  \cite{CDPR}. There, the authors considered \emph{chains of loops} with generic edge lengths to prove a tropical analogue of the well-known Brill-Noether theorem. We will discuss the parallels between chains of elliptic curves and tropical curves, as well as a generalization of our results to a tropical setting, in a forthcoming work.

\end{section}
\section{Degeneration methods and the genus of the {B}rill-{N}oether curve}\label{genus}

Degeneration methods on the moduli space of stable curves have been used by many authors. Eisenbud and Harris \cite{EH86} introduced limit linear series on reducible curves of compact type, a technique that allowed the use of reducible curves in the solution of several classical and novel problems. An account on this method and applications can be found in \cite{HM98}. See also \cite{Toh}.

Recently, Osserman \cite{Oss} generalized Eisenbud and Harris' work, constructing a proper moduli space of Eisenbud-Harris limit linear series for families of curves. Hence, in order to compute the genus and gonality of the one-dimensional Brill-Noether locus on a general curve, we will consider a family of curves over the spectrum of a discrete valuation ring such that the generic fiber is a generic nonsingular  curve  $C$ of genus $g$, while the central fiber $C_0$ is a generic chain as in 

 \begin{mydef}\label{chain} Let $C_1,\dots,C_g$ be elliptic curves each with points $P_i$, $Q_i$. Glue $Q_i$ to $P_{i+1},i = 1,\dots,g-1$. The genus of the resulting curve is $g$; we shall call such a curve a \emph{chain of elliptic curves} (of genus $g$). The chain is said to be general when the choice of pairs of points on the elliptic components is general.
\end{mydef}

In this first section, we will provide a new proof of the expression computing the genus of a Brill-Noether curve. Using degeneration techniques, it suffices to compute the genus of the space of limit linear series on a chain of elliptic curves.

\begin{teo}\label{genus} Let $a\in\mathbb{Z}_{\geq 1}$ and $C$ be a generic curve of genus $g=2a+1$. Then the genus of the Brill-Noether curve $W^1_{a+2}(C)$ is $$g_{W^1_{a+2}}=1+\frac{2a(2a+1)}{a+2}\cdot\mathfrak c_a,$$
where $\mathfrak c_a$ is the $a$-th Catalan number $$\mathfrak c_a=\frac{1}{a+1}{2a\choose a}.$$

\end{teo}

Using degenerations, it will suffice to compute the genus of the scheme of limit linear series on a chain of elliptic curves.

\begin{lema}\label{comps}
Let $C_0$ be a general chain of elliptic curves as above. 
 Then, the curve of limit linear series $W^1_{a+2}(C_0)$ is a reducible
  curve with $\nu=g\mathfrak c_a$ irreducible components each isomorphic to one 
  components of  $C_0$.
\end{lema}

\begin{proof}
    A limit linear series of degree $d$ and (projective) dimension one 
is given by a line bundle of  degree $d$ on each component
 along with a two-dimensional subspace of sections of this line bundle so that, at the nodes, corresponding 
 sections vanish with order adding at least to $d$. Denote by $u^i_1<u^i_2$ the order of vanishing of the  sections
 in this subspace at the point $P_i$ and by  $v^i_1>v^i_2$ the order of vanishing of the  sections
 in this subspace at the point $Q_i$. This means that all the sections on the subspace vanish to order 
 at least $u^i_1$ at $P_i$ and $v^i_2$ at $Q_i$ while there is at least one section $s_i$ vanishing to order 
 $u^i_2$ at $P_i$ and at least one section $s'_i$ vanishing to order at least $v^i_1$ at  $Q_i$. In particular,
 $s_i $ vanishes to order 
 $u^i_2$ at $P_i$ and $v^i_2$ at $Q_i$. As the total order of vanishing of a section of a line bundle
  cannot be larger
  than the degree of the line bundle, $u^i_2+v^i_2\le d$ and equality implies that the line bundle
   on the curve $C_i$ is ${\mathcal O}(u^i_2 P_i +v^i_2Q_i)={\mathcal O}(u^i_2 P_i +(d-u^i_2)Q_i)$. Similarly, 
   $u^i_1+v^i_1\le d$ and equality implies that the line bundle
   on the curve $C_i$ is ${\mathcal O}(u^i_1 P_i +(d-u^i_1)Q_i)$. Moreover, as $P_i$ and $Q_i$ are generic points and 
$u^i_1\not= u^i_2$, 
   ${\mathcal O}(u^i_1 P_i+ (d-u^i_1)Q_i)$ is not isomorphic to $ {\mathcal O}(u^i_2 P_i +(d-u^i_2)Q_i)$. 
   This means that only one of the above equalities can hold. 
   Therefore, $u^i_1+v^i_1+u^i_2+v^i_2\le 2d -1$ with equality if and only if 
   either the line bundle on the curve $C_i$ is 
   ${\mathcal O}(u^i_1 P_i +(d-u^i_1)Q_i)$ or $ {\mathcal O}(u^i_2 P_i +(d-u^i_2)Q_i)$. That is 
   $u_1^i+v_1^i+u_2^i+v_2^i\le 2d -1-\epsilon _i$ where $\epsilon _i=0$ if the line bundle on the
   $i^{th}$ component is of the form given above and is 1 otherwise. Therefore, 
   $$\sum _i (u^i_1+v^i_1+u^i_2+v^i_2)\le (2d -1)g-\sum \epsilon _i.$$

Let us now consider the vanishing at the nodes. By definition of limit linear series,
    $u^{i+1}_1+v^i_1\ge d, u^{i+1}_2+v^i_2\ge d$ for each node $i=1,\cdots, g-1$, while $u^1_1+u^1_2\ge 1, 
    v^g_1+v^g_2\ge 1$. Therefore, 
    $$\sum _i (u^i_1+v^i_1+u^i_2+v^i_2)\ge 2(g-1)d+2.$$
    
From the two inequalities, $$\sum_i\epsilon _i\le 2d-g-2.$$

\begin{rem} Here we have just been proving the Brill-Noether Theorem for our curve when $r=1$. The same proof works 
    for any other value of $r$ greater than one. The proof does not require characteristic zero.
     Moreover, a simplified proof of Petri is also possible in the same manner and in all 
     characteristics using these curves (see \cite{W} and \cite{CLT}).
     \end{rem}
     
     As in our case $\rho(g,r,d)=2d-g-2=1$, we find that 
     for a linear series,  the line bundle on each component but one is predetermined. This shows that the Brill-Noether curve 
is reducible. A component 
     corresponds to the choice of one of the components of the original curve  in which 
     the line bundle is free to vary and the choice,  on the remaining  $g-1$  components of 
     the original curve, of one of the two line bundles: $ {\mathcal O}(u^i_1 P_i +(d-u^i_1)Q_i)$ or  ${\mathcal O}(u^i_2 P_i +(d-u^i_2)Q_i)$. 
     We will distinguish each of these 
possibilities by the subindex of the order of vanishing that appears in the presentation. 
That way, the limit linear series is determined by choosing one index (either 1 or 2) on each of $g-1$ 
 components, together with an arbitrary line bundle of degree $d$ on the remaining component.
      As the set of line bundles on an elliptic curve is isomorphic to the curve itself,
       this component of the Brill-Noether curve 
      is isomorphic to the component of the original curve on which the line bundle is arbitrary.
      
      Let us see now which choices of subindices (1 or 2) are allowable. 
       If on one component $j$
      we choose the first subindex, that is, the line bundle is $ {\mathcal O}(u^j_1 P_j +(d-u^j_1)Q_j)$, then  $v^j_1=d-u^j_1$, while 
$v^j_2=d-u^j_2-1$.
      Therefore, $u^{j+1}_1=d-v^j_1=u^j_1$, while $u^{j+1}_2=d-v^j_2=u^j_2+1$. Similarly, if we choose instead the second 
subindex for the line 
      bundles on the component $j$, we will have $u^{j+1}_1=u^j_1+1$, while $u^{j+1}_2=u^j_2$. 
      
      It follows from the above discussion (using that   $u^1_1=0, u^1_2=1$) that 
      $u^i_1$ is the number of times we did not choose the line bundle to be 
       $ {\mathcal O}(u^j_1 P_j +(d-u^j_1)Q_j)$ on the components $j<i$, while $u^i_2$
       is one more than the number of times we did not choose the line bundle to be 
       $ {\mathcal O}(u^j_2 P_j +(d-u^j_2)Q_j)$ on the components $j<i$. This number of ``non-chosen" components
       includes the case in which we choose the line bundles on one of these components to be general. 
       As $u^i_1<u^i_2$ and  $u^1_1=0, u^1_2=1$, this means that, on any subcurve starting at $C_1$, we cannot choose the line 
bundle to be 
       $ {\mathcal O}(u^j_2 P_j +(d-u^j_2)Q_j)$ more times than we choose it to be  $ {\mathcal O}(u^j_1 P_j +(d-u^j_1)Q_j)$.
       
This problem can be easily formulated in combinatorial terms using the subindices in the expressions of the line bundles: we 
have to choose occurrences of the numbers 1 and 2 on a list/array of length $2a$ so that both 1 and 2 appear $a$ times each, and 
the number of occurrences of 2 up to and including any position on the list is always less than or equal to the number of 
occurrences of 1 up to that same position. We will call an array satisfying these conditions \emph{admissible}.
       
We are interested in counting how many of these choices for the subindices (line bundles) we have. The solution to this 
problem is the Catalan number $\mathfrak c_a$ (see \cite[Prop. 6.2.1]{Sta99}).
 \end{proof}

Note that we can describe the components of the Brill-Noether curve
in terms of  paths in a 2-dimensional lattice from $(0,0)$ to $(a,a)$ never rising above the diagonal 
with steps $(0,1)$ and $(1,0)$ (that we will identify with subindices 1 and 2 in 
Lemma \ref{comps}, respectively) along with a marked vertex\footnote{These paths, without the marked point, are classically known as \emph{Dyck paths} (cf. \cite[Cor. 6.2]{Sta99}).}. The vertex will represent the $i$th component of the original curve where the bundle is generic. See Figure \ref{fig:M1}.

 \begin{figure}
\centering
\begin{tikzpicture}[scale=0.8]
(0,0) rectangle +(5,5);
\draw[help lines] (0,0) grid +(5,5);
\draw[dashed] (0,0) -- +(5,5);
\coordinate (prev) at (0,0);
\draw [color=blue!60, line width=2] (0,0)--(1,0)--(1,1)--(2,1)--(2,2)--(4,2)--(4,3)--(5,3)--(5,5);12121L12122
\draw [color=red!60, line width=1] (0,0)--(1,0)--(1,1)--(2,1)--(2,2)--(4,2)--(4,3)--(5,3)--(5,5);  121211L2122
\draw (3,2) node [scale=0.5, circle, draw,fill=blue!60]{};
\draw (4,2) node [scale=0.5, circle, draw, fill=red!60]{};
 \end{tikzpicture}
\caption{} \label{fig:M1}
\end{figure}

The arguments above provide a new way to compute the classical Castelnuovo numbers (cf. \cite[VII, Theorem 4.4]{ACGH}).

\begin{teo}\label{castelnuovo} The number of $g^1_{a+1}$ on a curve of genus $2a$ is given by the Catalan number $$\mathfrak c_a=\frac{1}{a+1}{2a\choose a}.$$

More generally, the number of $g^r_{r(a+1)}$ on a curve of genus $a(r+1)$ is given by the generalized Catalan number 
$$\mathfrak C_{a,r+1}=(a(r+1))!\prod_{i=0}^{r}\frac{i!}{(a+i)!}.$$ 
 \end{teo}
\begin{proof}
 The same argument that we used to find the number of components of the Brill-Noether curve could be used to
 describe   $W^r_{r(a+1)}$ on a curve of genus $a(r+1)$. In this case the Brill-Noether locus is finite. On our reducible curve, its points
 correspond to the number of ways to order $a$ copies of each of  the numbers $1,\dots,r+1$ so that at each position in the corresponding array a larger number does not appear more times than a smaller number. The number of arrays of this form is given by the \emph{generalized Catalan numbers}. These numbers, that could be considered a generalization of classical Catalan numbers, arise as the solution to the many-candidate ballot problem in combinatorics (see \cite{Z}). 
\end{proof}

Let us now find the number of nodes of the Brill-Noether curve.

\begin{lema}\label{nodes}
 The  Brill-Noether curve $W^1_{a+2}\left(C_0\right)$ (where $C_0$ is as in Lemma \ref{comps}), is a nodal curve with 
 $$\delta= 2\left((2a+1)\mathfrak c_a-\mathfrak c_{a+1}\right)$$ 
 nodes.
\end{lema}

\begin{proof} Here we  count  the number of points of intersection of the components described above.
We want to find conditions for two of the components of the Brill-Noether curve to intersect, where a component of the 
degenerated Brill-Noether curve is an elliptic curve as described in \ref{comps}. 
 We will see in the next Lemma that there are no triple or worse points.
 Then in Lemma \ref{nodal}, we check  that the points of intersection on a given component are all different
  and that the difference of two points of intersection is a multiple of $P_i-Q_i$, where $P_i, Q_i$ are the nodes on the component $C_i$ of the original curve.

Let $X$  be a component of $W^1_{a+2}(C_0)$ corresponding, via the description in Lemma \ref{comps}, 
to the ordering of 1's and 2's given by the sequence 
$\underline{\alpha}=(\alpha_1,\dots,\alpha_{2a})$ satisfying the conditions mentioned there 
and with arbitrary line bundle 
on the $i$th component. Similarly, assume that a second component $X'$ corresponds to the sequence $
\underline{\alpha}'=(\alpha'_1,\dots,\alpha'_{2a})$, with arbitrary line bundle on the $i'$th component, 
with $i'>i$.

As the line bundle on each $C_j$ other than $C_i$ (resp. $C_{i'}$) is determined by the sequence $\underline{\alpha}$, it follows 
that \begin{equation}\label{paths}
\left.\begin{aligned}\alpha_1&=\alpha'_1,\dots, \alpha_{i-1}=\alpha '_{i-1}   \\
 \alpha_{i'}&=\alpha'_{i'},\dots,  \alpha_{2a}=\alpha'_{2a}.\end{aligned}
 \right.
 \qquad
      \end{equation} 
      
If $i'=i+1$, the sequence $\underline\alpha$ equals $\underline{\alpha}'$. Conversely, 
if  $\underline\alpha$ equals $\underline{\alpha}'$ and $i'=i+1$, the two components intersect:
 the line bundle in the first $i-1$ and the last $g-i-1$ components of the original curve is
 the same for both choices. The line bundles on the components $i$ and $i+1$ are free to vary for
  one of the choices and completely determined for the other. Hence, there exists a unique point 
  of intersection between these two components of the Brill-Noether curve. 
   That is, each component of the Brill-Noether curve 
corresponding to a choice of an ordering $\underline\alpha$ and a vertex $i, 1\le i\le g-1$ of the chain of elliptic
 curves intersects the 
component corresponding to the same choice of $\underline\alpha$ and vertex $i+1$.
 The number of these intersections is $(g-1)\mathfrak c_a$. 
 Alternatively, one could think of each choice of  sequence $\underline\alpha$ as giving rise to a component of genus $g$ of the Brill-Noether curve.
 
 Assume now that $i'\neq i+1$. As the line bundle on $C_i$ for a point in $X$ is generic, we have 
 $u^{i+1}_1=u^{i}_1+1$ and $u^{i+1}_2=u^{i}_2+1$. On the second component of the Brill-Noether curve $X'$, however, the index 
corresponding to $\alpha '_i$ does not increase. 
 Since the line bundle on the curve $C_{i+1}$ must be the same for one point in $X$ and one point in $X'$ (we are assuming the 
components intersect),
  this implies $\alpha' _{i+1}\not= \alpha' _{i}$.
  Recall now that the line bundle on the component $C_i$ is generic for the points of $X$. On the elliptic components of $X$ after $C_i$ the 
line bundles are again determined by $\underline\alpha$. (Note that the indices in $\underline\alpha$ are now off by one,  given that  the sequence $\underline\alpha$ carries no information for the elliptic component $C_i\subset X$.) We deduce then that 
$\alpha_i=\alpha' _{i+1}$. 
  As the discrepancy in vanishing between points in $X$ and points in $X'$ persists for the index $\alpha '_i\not=\alpha' _{i+1}=
\alpha_{i}$, we have 
  $$\alpha' _{i+1}=\alpha' _{i+2}=\cdots =\alpha' _{i'-1}=\alpha _{i}=\alpha _{i+1}=\cdots =\alpha_{i'-2}.$$ Then, necessarily, $
\alpha_{i'-1} =\alpha'_i$.
  
For example, if $\alpha _i=1$, this would say that \begin{align*}\alpha _{i}=\alpha _{i+1}=\cdots =\alpha_{i'-2}=1, \\ \alpha' _{i+1}=\alpha' _{i+2}=\cdots =\alpha' _{i'-1}=1,\end{align*} while 
  \begin{align*}\alpha '_i=2, \alpha_{i'-1}=2.\end{align*}

Intersections of two components $X$ and $X'$ of the Brill-Noether curve, in this case, correspond to lattice paths given by sequences $\underline\alpha$ and $\underline\alpha'$, respectively, whose entries satisfy (\ref{paths}), i.e. the paths are the same up to the first marked vertex and after the second marked vertex. After the first vertex, the paths become parallel, since $\alpha'_i=2$, and meet again on the second vertex, since $\alpha_{i'}=2$. 

\begin{ej}
The lattice path for the intersection of the components $X,X'\subset W_{7}^1(C_0)$ given by $\underline\alpha=\underline{\alpha}'=(1,2,1,2,1,1,2,1,2,2)$, with generic bundles for $i=6$ and $i'=7$, respectively, is shown in Figure \ref{fig:M1}.
\end{ej}

\begin{ej} An example of an intersection of lattice paths corresponding to the intersection of components of the Brill-Noether curve in this case is represented in Figure \ref{fig:M2}. There, the first component $X$ is given by the sequence $\underline\alpha=(1,1,2,1,1,1,2,2,2,2)$ and the marked vertex for $i=4$; the second intersecting component is given by $\underline\alpha'=(1,1,2,2,1,1,1,2,2,2)$ and $i'=8$.
\end{ej}
 
We are interested in counting how many of these intersections we have. For each pair (path, marked vertex) one can construct a lattice path corresponding to an intersection with the given pair, except when the marked vertex lies on the diagonal. The total number of these pairs is $(a-1)\mathfrak c_a$, regardless of where the vertex is. The number of paths when the marked vertex is on the diagonal is the product of the number of paths from $(0,0)$ to this vertex and the number of paths from the vertex to $(a,a)$, i.e. the product of Catalan numbers $\mathfrak c_k \mathfrak c_{a-k}$. The number of nodes on the curve in this case is therefore \begin{equation}\label{eq:cortes}(a-1)\mathfrak c_a-\sum_{k=1}^{a-1} \mathfrak c_k \mathfrak c_{a-k}.\end{equation}

\begin{figure}
\begin{floatrow}
\ffigbox{\caption{} \label{fig:M2}}{\exedout}
\ffigbox{\caption{}\label{fig:M3}} {\exedouttwo}
\end{floatrow}
\end{figure}

A similar argument is valid for $\alpha_i=2$, yielding again an expression equal to (\ref{eq:cortes}). Hence the formula for the number of nodes is

\begin{align*}
\delta&=2\left((a-1)\mathfrak c_a-\sum_{k=1}^{a-1} \mathfrak c_k \mathfrak c_{a-k}\right)+(g-1)\mathfrak c_a.
\end{align*}

From the well-known recursion of Catalan numbers $$\mathfrak c_{a+1}=\sum^{a}_{k=0}\mathfrak c_k \mathfrak c_{a-k},$$ and the fact that $\mathfrak c_0=1$, we obtain the expression $\delta= 2\left((2a+1)\mathfrak c_a-\mathfrak c_{a+1}\right)$ for
the number of nodes of the limit of the Brill-Noether curve, $W^1_{a+2}\left(C_0\right)$. \end{proof}

\begin{lema}\label{nodal} Let $C_0$ be  general  a chain of elliptic curves. Then, 

\begin{enumerate} 
\item[(a)]  The points at which a given component $X$ of the  Brill-Noether curve $W^1_{a+2}\left(C_0\right)$ intersects any other of the components of $W^1_{a+2}\left(C_0\right)$ are all different.
  
\item[(b)] A fixed  component $X$ has at most four points of intersection with other components.
  
\item[(c)] Assume that  $X$ is isomorphic to a component  $C_i$ of $C_0$ and $P_i, Q_i$ are the nodes of $C_i$.
   If $R, S$ are two points on $X$ that are points of intersection with other components of the Brill-Noether
   curve, then $R-S=k(P-Q)$ for some integer $k$.
\end{enumerate}
\end{lema}

\begin{proof} A component $X$ of the Brill-Noether curve of $C_0$ corresponds to the choice of a sequence 
$\underline{\alpha}=(\alpha_1,\dots,\alpha_{2a})$ of 1's and 2's satisfying the conditions in Lemma \ref{comps} and with arbitrary line bundle on the $i$th component. 
Assume we have 
$$\alpha_{t-1}\not=\alpha_{t}= \alpha_{t+1}=\cdots =\alpha_{i-1},$$
$$\alpha_{i}=\alpha_{i+1}=\cdots =\alpha_{s-1}\not=\alpha_{s}.$$
Then this component intersects the following components 
\begin{enumerate}
\item[(1)] The component corresponding to the same chain $\underline{\alpha}$ 
with arbitrary line bundle on the $(i-1)$th component (assuming that $i>1$).
\item[(2)] The component corresponding to the same chain $\underline{\alpha}$ 
with arbitrary line bundle on the $(i+1)$st components (assuming that $i<2a$).
\item[(3)]  The component corresponding
to the chain 
$$\underline{\alpha}'=(\alpha_1,\dots,\alpha_{i-1},\alpha_{s}, \alpha_{i},\alpha_{i+1},
\dots,\alpha_{s-1},\alpha_{s+1},\dots, \alpha_{2a})$$
with arbitrary line bundle on the component $s+1$ (assuming that the chain $\underline{\alpha}'$ is admissible). 
\item[(4)] The component corresponding to the chain 
$$\underline{\alpha}''=(\alpha_1,\dots,\alpha_{t-2}, \alpha_{t},\dots ,\alpha_{i-1},\alpha_{t-1}, \alpha_{i},\dots, \alpha_{2a})$$
with arbitrary line bundle on the component $t-1$ (assuming that the chain $\underline{\alpha}''$ is admissible).
\end{enumerate}
 
 Choose now another   component $X'$ of the Brill Noether curve corresponding to an admissible sequence $\hat{ \underline{\alpha} }$ and the choice of a component 
 $j\not= i$ on which the 
 line bundle is free to vary. Denote by $\hat u_1<\hat u_2$, the vanishing at $P_i$ of a linear series in $X'$. 
 Then, the restriction to $C_i$ of the line bundle coming from a linear series on $X'$ is  
 $${\mathcal O}({{\hat u_{\alpha_i}}}P_i+(d-{{\hat u_{\alpha_i}}})Q_i)\  {\rm if} \ j>i \ \ \mbox{and } {\mathcal O}({{\hat u_{\alpha_{i-1}}}}P_i+(d-{{\hat u_{\alpha_{i-1}}}})Q_i)\  {\rm if} \ j<i.$$
 It follows that the line bundle on the component $C_i$ in the cases listed above is 
\begin{enumerate}
\item[(1)] ${\mathcal O}((u_{\alpha_{i-1}}+1)P_i+(d-( u_{\alpha_{i-1}}+1))Q_i)\ $
 (assuming that $i>1$).
\item[(2)] ${\mathcal O}(u_{\alpha_{i}}P_i+(d- u_{\alpha_{i}})Q_i)  $
 (assuming that $i<2a$).
\item[(3)]  ${\mathcal O}(u_{\alpha_{s}}P_i+(d- u_{\alpha_{s}})Q_i)  $  if $\underline{\alpha}'$ is admissible where 
$$\underline{\alpha}'=(\alpha_1,\dots ,\alpha_{i-1},\alpha_{s}, \alpha_{i},\alpha_{i+1},
\dots, \alpha_{s-1},\alpha_{s+1},\dots, \alpha_{2a}).$$
    The condition for admissibility is automatically satisfied if $\alpha_s=1$, since 1 in an admissible sequence can always be pulled back to earlier entries. When $\alpha _s=2$,
 we need the number of ones in the partial sequence $(\alpha_1,\dots ,\alpha_{i-1})$  to be greater than the number of twos. Therefore, $u_i+1<u_2$. That is, 
 \bigskip
 \item[(*3)]\hspace{2.5cm} if $ \alpha _{s}=2, \  u_{\alpha _{i}}+1=u_1+1<u_2=u_{\alpha _{s}}$.
 \bigskip
 \item[(4)]  ${\mathcal O}((u_{\alpha_{t-1}}+1)P_i)+(d- u_{\alpha_{t-1}}-1)Q_i)  $
assuming that  the sequence 
$$\underline{\alpha}''=(\alpha_1,\dots,\alpha_{t-2}, \alpha_{t},\dots ,\alpha_{i-1},\alpha_{t-1}, \alpha_{i},\dots, \alpha_{2a})$$ is admissible.
 The condition for admissibility is automatically satisfied if $\alpha_{t-1}=2$ since 2 in an admissible sequence can always be pushed forward to later entries. When $\alpha _{t-1}=1$, 
 the sequence $(\alpha_{1}\dots ,\alpha_{t-2},\alpha_{t}\dots, \alpha_{i-1})$ needs to be the 
 beginning of an admissible sequence, and therefore satisfies that there are at least as many 1s as there are 2s. The vanishing at 
 $P_{i-1}$ after such list would be $u_1, u_2-1$. So when the condition in (4) is satisfied, $u_1<u_2-1$ if $\alpha _{t-1}=1$. That is, \bigskip
 \item[(*4)]\hspace{2.5cm}  if $\alpha _{t-1}=1, \  u_{\alpha _{t-1}}+1=u_1+1<u_2=u_{\alpha _{i-1}}$.
\end{enumerate} 
 \bigskip
 
 We now need to check that the four values 
 $$u_{\alpha_{i-1}}+1, u_{\alpha_{i}}, u_{\alpha_{s}}, u_{\alpha_{t-1}}+1$$
 are different whenever they appear. Note that possible subindices are either 1 or 2 and that, by definition, 
 $u_1<u_2$. We also have $\alpha _{i-1}\not= \alpha _{t-1}, \alpha _{i}\not= \alpha _{s}$. 
 The only overlaps that could happen are then if one of the entries $\alpha_{i-1}, \alpha_{t-1}$ 
 were 1 and one of $\alpha_{i}, \alpha_{s}$ were 2 and $u_2=u_1+1$.
  Conditions (*3) and (*4) exclude all of these options except for $\alpha_{i-1}=1, \alpha_{i}=2$
   and $u_2=u_1+1$. The condition $u_2=u_1+1$ tells us that in the sequence $(\alpha _1,\dots, \alpha_{i-1})$
   there are as many ones as there are twos. Therefore, we would not be able to include in this sequence the additional $\alpha_i=2$.
   
   This proves that all points of intersection of a given component of the Brill-Noether curve
    with any other component are different.
    
    The Brill-Noether curve lives in the Jacobian of the reducible curve, which is isomorphic to the product
     of elliptic curves $C_1\times\cdots\times  C_g$. Each component is the product of one fixed point on each of these $C_i$
     except for one in which it is the whole curve. So the intersections are transversal. Moreover,
     from the proof of the Gieseker-Petri theorem, all points other than points of intersection are non-singular. 
\end{proof}

\begin{proof}[Proof of Theorem \ref{genus}] On the special fiber, the Brill-Noether curve degenerates to a reducible curve with a certain number of irreducible components $\nu$. Let us denote by $C_i$ each of these components and let $g_{C_i}$ be their genera.

We have then (cf. \cite[2.14]{HM98}) \begin{align*}
g_{W^1_{a+2}}&=\sum_{i=1}^{\nu}g_{C_i}+\delta-\nu+1.\end{align*}

The result follows now from Lemmas \ref{comps} and \ref{nodes}.
\end{proof}

\section{Gonality of the Brill-Noether curve when $g=5$}\label{gonality}

The gonality $\mbox{gon}(C)$ of a smooth curve $C$ of genus $g$ is considered the second natural invariant of $C$, after its genus, and it is defined as
$$\mbox{gon}(C)=\{\deg f\mid f:C\rightarrow \mathbb P^1 \mbox{ is a surjective morphism}\}=\min\{d\in\mathbb Z_{>0}\mid C\mbox{ has a } g_d^1 \}.$$

The goal of this section is to show that, for a generic curve $C$ of genus 5, the Brill-Noether curve $W_4^1(C)$ has gonality 6. 

Let us first note that a generic curve $C$ of genus 5 is the complete intersection of three quadrics in $\mathbb{P}^4$. The singular quadrics containing the curve give a non-singular quintic in the ${\mathbb P}^2$ of all quadrics containing the curve.  The curve $W^1_4(C)$ is an unramified double cover of this plane quintic
(see \cite[p.207]{ACGH}  and \cite[Section III]{T}). From \cite[Section 3]{AF} an \'etale double cover of a curve of even genus has gonality less than the maximum. In our case, as $W^1_4(C)$ has  genus 11, this observation implies that $\mbox{gon}(W^1_4(C))\leq 6$. Our proof below will show that $W^1_4(C)$ is not 5-gonal and will therefore imply that $\mbox{gon}(W^1_4(C))$ is exactly 6.

Let us consider a family of curves so that the generic fiber is a generic smooth curve of genus 5 and whose central fiber is a chain of 5 elliptic curves as in Definition \ref{chain}. By a semicontinuity argument, it suffices to show that the Brill-Noether curve of the central fiber $C_0$ of this family is not 5-gonal. Note that in this case, the corresponding Brill-Noether curve $W_4^1(C_0)$ is reducible not of compact type. It suffices then to show (see proof of \cite[Theorem 5]{HM82}) that $W_4^1(C_0)$ does not admit an  admissible cover of degree 5.

\begin{mydef}[\cite{HM82,Oss05}] A degree $d$ map $f : C\rightarrow D$ of stable curves (with marked points) is defined to be an admissible cover of degree $d$ if:
\begin{enumerate}
\item[(i)] The set of nodes of $C$ are precisely the preimage under $f$ of the set of nodes of $D$.
\item[(ii)] The set of smooth ramification points of $C$ are the marked points of $C$.
\item[(iii)] Lift $f$ to $\tilde f : \tilde C \rightarrow \tilde D$ on the normalizations of $C$ and $D$. Then for each
node of $C$, the ramification indices of $\tilde f$ at the two points of $\tilde C$ lying above the node must coincide.
\end{enumerate}
\end{mydef}

Let $C$ be a chain of five elliptic curves. By the work in the previous section, the Brill-Noether curve $W_{4}^1(C)$ can be described in terms of 
the two admissible sequences $\underline\alpha'=(1,2,1,2)$ and $\underline\alpha''=(1,1,2,2)$ (see the proof of Lemma \ref{comps}).
  Each sequence will give rise to five components in the Brill-Noether curve isomorphic to each of the five components of the given curve $C$. 
  We summarize the possibilities in the following tables.

\begin{table}[h!]
\scalebox{0.8}{\begin{tabular}{c||c|c|c|c|c|}
 $g_4^1$& $C_1'$ &  \color{green!60!black} $C_2'$ & $C_3'$ &   \color{red!60!black} $C_4'$& $C_5'$ \\ \hline
 $$ & $L$ &  \color{green!60!black} $4Q$ & $4Q$ &   \color{red!60!black} $4Q$& $4Q$ \\ \hline
 1 & $P+3Q$ &  \color{green!60!black} $L$ & $2P+2Q$ &   \color{red!60!black} $2P+2Q$& $2P+2Q$  \\ \hline
 2 & $3P+Q$ &   \color{green!60!black}$3P+Q$ & $L$ &   \color{red!60!black} $P+3Q$& $P+3Q$  \\ \hline
 1 & $2P+2Q$ &  \color{green!60!black} $2P+2Q$ & $2P+2Q$ &   \color{red!60!black} $L$& $3P+Q$  \\ \hline
 2 & $4P$ &  \color{green!60!black} $4P$ & $4P$ &   \color{red!60!black} $4P$& $L$  \\ \end{tabular}
\ \ \ \ \ 
\begin{tabular}{c||c|c|c|c|c|}
 $g_4^1$&    $C_1''$ & \color{red!60!black} $C_2''$ & $C_3''$ &  \color{green!60!black}$C_4''$& $C_5''$ \\ \hline
  $$           &           $L$      &      \color{red!60!black}  $4Q$       &       $4Q$        &        \color{green!60!black} $4Q$      &        $4Q$        \\ \hline
 1               &        $P+3Q$  &       \color{red!60!black}    $L$         &       $4Q$       &     \color{green!60!black}    $4Q$      &     $4Q$       \\ \hline
 1 		&	 $P+3Q$ & \color{red!60!black} $P+3Q$ &	 $L$ & \color{green!60!black}$3P+Q$& $3P+Q$ \\ \hline
 2 & $4P$ & \color{red!60!black} $4P$ & $4P$ & \color{green!60!black}$L$& $3P+Q$ \\ \hline
 2 & $4P$ &  \color{red!60!black}$4P$ & $4P$ & \color{green!60!black}$4P$& $L$  \\ 
\end{tabular}}
\end{table}

The column in the far left indicates the combinatorial type mentioned above. 
Each of the remaining entries columns corresponds to one of the five components of the Brill-Noether curve. 
An entry $aP+bQ$  on the $i^{th}$ row means that the limit linear series on $C$ on the component $i$  corresponds to the line bundle $\mathcal {O}(aP_i+bQ_i)$.
 An entry assigned the value $L$ indicates that we are choosing an arbitrary line bundle on this component (and therefore this gives a one dimensional choice). 

Two components of the Brill-Noether curve intersect when three of the entries that are completely determined (that is, not an $L$) in both are identical. 
For example, the component $C'_2$ intersects $C'_1$ at the point $P+3Q$, $C'_3$ at the point
  $2P+2Q$, and $C''_4$ at the point $4Q$. Similarly,  the component $C''_2$ intersects $C''_1$ at the point $P+3Q$, $C''_3$ at $4Q$, and $C'_4$ at the point $2P+2Q$. 
  The component $C'_4$ intersects $C'_5$ at the point $3P+Q$, $C'_3$ at the point $2P+2Q$, and $C''_2$ at the point $4P$.
   Finally, the component $C''_4$ intersects $C''_5$ at the point $3P+Q$, $C''_3$ at the point $4P$, and $C'_2$ at the point $2P+2Q$.
 The remaining components of the Brill-Noether curve intersect only neighboring components in the same chain.
 
 For each of the components with three points of intersection with the rest,
  we denote by $X_i'$ (resp. $X''_{i}$) the line bundle of intersection of the component $C'_i$ (resp. $C''_{i}$) with the shortest tail of components of the same sort; 
  denote by $Y'_i$ (resp. $Y''_{i}$) the intersection with the longest tail,  and $Z_i'$ (resp. $Z_{i}''$) the intersection with the component $C_{6-i}''$ (resp. $C_{6-i}'$).
  Then $2X'_i=Y'_i+Z'_i$ and $2X''_i=Y''_i+Z''_i$. 
  As $P_i$ and $Q_i$ were generic on each component $C_i$ of the given curve, also $Y_i$ and $Z_i$ are generic on each $C_i', C_i'',$ $i=2,4$. 
  Similarly, the points of intersection of $C_i', C_i'',$ $i=1, 3, 5 $ with the remaining components are generic too.

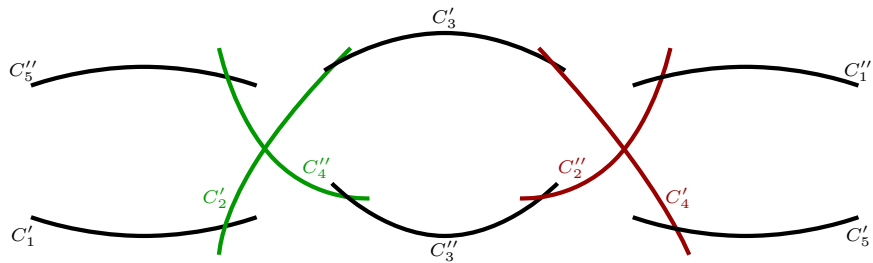
\begin{figure}
\begin{tikzpicture}[scale=1]
\node [color=black]  at  (-7.1,2.2) {\tiny$ C''_5$};
\node [color=green!60!black]  at  (-3.2,.9) {\tiny$ C''_4$};
\node [color=black]  at  (-1.5,-.2) {\tiny$ C''_3$};
\node [color=red!60!black]  at  (.2,.9) {\tiny$ C''_2$};
\node [color=black]  at  (4,2.2) {\tiny$ C''_1$};
\draw [black,  ultra thick] plot [smooth, tension=1] coordinates { (-7,2) (-5.5,2.25) (-4,2)};
\draw [green!60!black,  ultra thick] plot [smooth, tension=1] coordinates { (-4.5,2.5) (-3.75,1) (-2.5,0.5)};
\draw [black,  ultra thick] plot [smooth, tension=1] coordinates { (0,.7) (-1.5,0) (-3,.7)};
\draw [red!60!black,  ultra thick] plot [smooth, tension=1] coordinates { (-.5,0.5) (.75,1) (1.5,2.5)};
\draw [black,  ultra thick] plot [smooth, tension=1] coordinates { (1,2) (2.5,2.25) (4,2)};

\node [color=black]  at  (-7.1,0) {\tiny$ C'_1$};
\node [color=green!60!black]  at  (-4.55,.5) {\tiny$ C'_2$};
\node [color=black]  at  (-1.5,2.9) {\tiny$ C'_3$};
\node [color=red!60!black]  at  (1.6,.5) {\tiny$ C'_4$};
\node [color=black]  at  (4,0) {\tiny$ C'_5$};
\draw [black,  ultra thick] plot [smooth, tension=1] coordinates { (-7,0.25) (-5.5,0) (-4,0.25)};
\draw [green!60!black,  ultra thick] plot [smooth, tension=1] coordinates { (-4.5,-0.25) (-4,1) (-2.75,2.5)};
\draw [black,  ultra thick] plot [smooth, tension=1] coordinates { (0.1,2.2) (-1.5,2.7) (-3.1,2.2)};
\draw [red!60!black,  ultra thick] plot [smooth, tension=1] coordinates { (1.75,-0.25) (1,1) (-0.25,2.5)};
\draw [black,  ultra thick] plot [smooth, tension=1] coordinates { (1,0.25) (2.5,0) (4,0.25)};



%
%
%

\end{tikzpicture}
\caption{The Brill-Noether curve $W_4^1(C_0)$.}
\label{fig:x+3}
\end{figure}

  Consider now the circuit consisting of the curves $C_2', C_3', C_4', C_2'', C_3'', C_4''$. 
  For ease of notation, we relabel the curves as $D_1,\dots, D_6$. Note that $D_i$ is glued to  $D_{i-1}$ and $D_{i+1} $ by identifying $A_i\in D_i$ with $B_{i-1}\in D_{i-1}$,
  $i=1,\dots, 6$, where the indices are understood modulo 6. Moreover, the pair of points $A_i, B_i$ is generic on each component $D_i$

%
%
%
%
%
%
%

    Denote by $D$ the curve that contains the circuit and possibly a few rational tails and gives rise to an admissible cover $\pi: D\to R$.

  \begin{claim}\ 
  \begin{enumerate}
  \item There are no admissible covers of degree three of the circuit to a rational curve.
   \item An admissible cover of degree 4 of the circuit onto a rational curve $R$ is ramified at at least 2 nodes and every elliptic component $D_i$ of $D$
   maps with degree  two on its image.
    \item  An admissible cover of degree 5 of the circuit onto a rational curve $R$ is ramified at at least 1 node.   
  \end{enumerate}
  \end{claim}

  \begin{proof}[Proof of the claim] A map of an elliptic curve to a rational curve has degree at least two. Assume that $D_i$ maps to a rational component $R_i$ of $R$ 
  and that $P$ is a node of $R$ lying on $R_i$ at which the map $D_i\stackrel{\pi}{\to} R_i$ is not ramified. Let $R_j$ be the other component of $R$ gluing with $R_i$ at $P$.
  There are two (or more if the map of $D_i$ to $R_i$ has degree greater than two) points of $D_i$ mapping to $P$,
  therefore, there are (at least) two branches of $D$ gluing with these two points. As $R$ is connected, every component of $R$ can be joined to the component $R_i$
  with a tree. Therefore, every elliptic component not ramified over any node contributes two units to the degree of the map over every rational component $R_k$ of $R$.
  Because every elliptic curve is joined to two more elliptic curves, six of these contributions are double counted. So the degree of the map is at least 
 $$6\times 2-6-\delta =6-\delta,$$
  where $\delta$ is the number of nodes over which the map is ramified and equality would imply that all the maps $D_i\to R_i$ have degree two.
  
  If the admissible cover has degree at most three, then $6-\delta\le 3$.
  By the generality of the pair of points on a given elliptic component, there cannot be two ramification points on a single component. 
   As there are six components and  each node is common to two components,  there cannot be more than 3 ramification nodes for the map. 
  Hence, if the degree of the admissible covering is at most three, 
  the degree of the map of each elliptic component over the corresponding $R_i$ is two and $\delta$ is three. 
  Up to contracting a few rational components that glue to the rest at only one point in both $D$ and $R$, we can assume that there is one elliptic component
  $D_i$ mapping to a rational component $R_i$ that intersects the rest of $R$ at only one point. Then, both nodes of $D_i$ map to the same point on $R$. 
  As the map $D_i\to R_i$ is of degree two, it cannot be ramified at these points and we have a contradiction with the assumption that every component has a ramification node.
  This shows that there are no admissible covers of degree three.
  
 The proof that  the admissible cover of degree four and five have  $\delta\ge 2$  and $\delta\ge 1$ follows similarly.
  \end{proof}
  
  Assume now that we have an admissible cover of degree at most five on the whole Brill-Noether curve.
   Restricting to the circuit and contracting suitable rational components attached at a single point,  we obtain an admissible cover of the circuit onto a rational curve. 
   
  Assume that the degree of the admissible cover restricted to the circuit is four. Then, there are at least two ramification points on the circuit. 
 Consider now one of the tail components of the Brill- Noether curve (a curve not on the circuit). It is attached to the curve say $C_i'$ (or $C''_i$) 
 in the circuit at one point  $X_i\in C'_i$
  where $X_i$ satisfies  $2X_i=Y_i+Z_i$ and $Y_i, Z_i$ are the points of intersection with the rest of the circuit and are generic points on $C'_i$.
  As the admissible cover is ramified at one of the nodes in the circuit, the map of the elliptic curve to the rational curve 
   is given by the linear series either $|2Y_i|$ or $|2Z_i|$. As $2X_i\not\equiv 2Y_i$ and $2X_i\not\equiv 2Z_i$,
    the map cannot be ramified at $X_i$. Hence, the addition of that component adds degree one to the admissible cover of the circuit.
  As we have at least two such components, this would bring the degree from 4 to 6 contradicting the assumption that the degree of the admissible covering is  at most five.
  
    If the admissible cover restricted to the circuit has degree 5, a similar argument using the one ramification point in the circuit applies.
    
    \begin{rem} It is not hard to construct an admissible cover of degree 6 of the Brill-Noether curve onto a rational curve: consider two rational non-singular curves $R_1, R_2$
    and glue them by identifying $Q_1\in R_1$  with $P_2\in R_2$. Map $ C_2', C_4', C_3''$ to $R_1$ and  $ C_3', C_4'', C_1''$ to $R_2$. 
 In each case,    use the complete linear series $|Y+Z|$ determined by the two points that join these curves to each other (for simplicity of notation, we omit subindices and superindices). 
 We can adjust the maps so that the  nodes on each curve in the circuit map to the 
    node in the rational base. Note then, that in the curves  $ C_2', C_4',  C_4'', C_2'',$ the additional node $X$ satisfies $|2X|=|Y+Z|$. So $X$ is a ramification point of the
    map to the corresponding rational curve.
    Add 4 rational components $R'_1, R'_5, R''_1, R''_5$ at the images of $X_2', X_4', X_2'', X_4''$ and map $C'_1, C'_5, C''_1, C''_5$ respectively to these curves 
    using the linear series given by twice the node at these four curves. 
     Add then 16 more rational components to the Brill Noether curve at the remaining 16 points lying over $R_1'\cap R_1, R_5'\cap R_1, R_1''\cap R_2, R_5''\cap R_2$, 
     the four new nodes of the base. These 16 rational components map bijectively in groups of four over   $R'_1, R'_5, R''_1, R''_5$.
     One obtains an admissible cover of degree 6.
    \end{rem}

    \begin{figure}
\begin{tikzpicture}[scale=.75]


\node [color=black]  at  (1.5,7.7) {\footnotesize$ C'_3$};
\draw [black,  ultra thick, ] plot [smooth, tension=3] coordinates { (0,7) (5,7.5) (-0.2,6)};
\node [color=black]  at  (-4,5.3) {\footnotesize$ C'_4$};
\draw [black,  ultra thick, ] plot [smooth, tension=3] coordinates { (0,6) (-6,6.5) (0.2,5)};
\node [color=black]  at  (3.3,4.6) {\footnotesize$ C''_2$};
\draw [black,  ultra thick, ] plot [smooth, tension=3] coordinates { (0,5) (4,5.5) (-0.2,4)};
\node [color=black]  at  (-1.4,2.9) {\footnotesize$ C''_3$};
\draw [black,  ultra thick, ] plot [smooth, tension=3] coordinates { (0,4) (-6,4.5) (0.2,3)};
\node [color=black]  at  (5.2,3.1) {\footnotesize$ C''_4$};
\draw [black,  ultra thick, ] plot [smooth, tension=3] coordinates { (-.1,2.9) (5,3.5) (-0.3,1.9)};
\node [color=black]  at  (-3,9.1) {\footnotesize$ C'_2$};
\draw [black,  ultra thick, ] plot [smooth, tension=4] coordinates { (0.2,2) (-7,6.5) (0.2,7)};

\node [color=black]  at  (5,6.6) {\footnotesize$ C''_1$};
\draw [black,  ultra thick, ] plot [smooth, tension=3] coordinates { (5.5,6.3) (4,5.5) (5.5,5.5)};
\node [color=black]  at  (1.5,1.2) {\footnotesize$ C''_5$};
\draw [black,  ultra thick, ] plot [smooth, tension=3] coordinates { (1.5,1.5) (2,2.3) (2.5,1.85)};
\node [color=black]  at  (-3,7.3) {\footnotesize$ C'_5$};
\draw [black,  ultra thick, ] plot [smooth, tension=3] coordinates { (-2.5,7.5) (-2,6.4) (-1.3,7.4)};
\node [color=black]  at  (-6,9.2) {\footnotesize$ C'_1$};
\draw [black,  ultra thick, ] plot [smooth, tension=3] coordinates { (-5.5,9.8) (-5,9.15) (-4.5,9.8)};

\node [color=yellow!80!black] at (-5,9.1) {$\bullet$};
\node [color=green!75!black] at (-2,6.4) {$\bullet$};
\node [color=blue!80!white] at (4,5.5) {$\bullet$};
\node [color=orange!65!white] at (2,2.3) {$\bullet$};


\node [color=black]  at  (-6.2,-2.2) {\tiny$ R_2$};
\node [color=black]  at  (6.2,-2.2) {\tiny$ R_1$};
\draw [black,  ultra thick] plot [smooth, tension=1] coordinates { (-6,-2) (.5,-1)};
\draw [black,  ultra thick] plot [smooth, tension=1] coordinates { (6,-2) (-.5,-1)};

\node [color=black]  at  (-5.2,-3) {\tiny$ R'_1$};
\draw [black,  ultra thick] plot [smooth, tension=1] coordinates { (-5.2,-1.25) (-4.6,-3.15)};

\draw [black,  ultra thick] plot [smooth, tension=1] coordinates { (-5.3,5.3) (-4.7,6.45)};

\draw [black,  ultra thick] plot [smooth, tension=1] coordinates { (-5.3,6.2) (-4.7,7.45)};

\draw [black,  ultra thick] plot [smooth, tension=1] coordinates { (-5.3,4.2) (-4.7,5.3)};

\draw [black,  ultra thick] plot [smooth, tension=1] coordinates { (-5.3,3.4) (-4.7,4.5)};
\node [color=yellow!80!black] at (-5,6.8) {\tiny$\bullet$};
\node [color=yellow!80!black] at (-5,5.95) {\tiny$\bullet$};
\node [color=yellow!80!black] at (-5,4.75) {\tiny$\bullet$};
\node [color=yellow!80!black] at (-5,3.9) {\tiny$\bullet$};
\node [color=yellow!80!black] at (-5,-1.85) {$\bullet$};
\node [color=black]  at  (-2,-2.5) {\tiny$ R'_5$};
\draw [black,  ultra thick] plot [smooth, tension=1] coordinates { (-2.2,-.5) (-1.6,-2.4)};
\draw [black,  ultra thick] plot [smooth, tension=1] coordinates { (-2.2,8.8) (-1.6,7.5)};
\draw [black,  ultra thick] plot [smooth, tension=1] coordinates { (-2.2,4.8) (-1.6,3.5)};
\draw [black,  ultra thick] plot [smooth, tension=1] coordinates { (-2.2,3.8) (-1.6,2.5)};
\draw [black,  ultra thick] plot [smooth, tension=1] coordinates { (-2.2,2.4) (-1.6,1.3)};
\node [color=green!75!black] at (-2,8.33) {\tiny$\bullet$};
\node [color=green!75!black] at (-2,4.38) {\tiny$\bullet$};
\node [color=green!75!black] at (-2,3.3) {\tiny$\bullet$};
\node [color=green!75!black] at (-2,1.95) {\tiny$\bullet$};
\node [color=green!75!black] at (-1.9,-1.4) {$\bullet$};

\node [color=black]  at  (2.75,-2.1) {\tiny$ R''_5$};
\draw [black,  ultra thick] plot [smooth, tension=1] coordinates { (2.15,-.8) (1.6,-2.75)};
\draw [black,  ultra thick] plot [smooth, tension=1] coordinates { (2.3,8) (1.63,6.7)};
\draw [black,  ultra thick] plot [smooth, tension=1] coordinates { (2.3,7) (1.63,5.7)};
\draw [black,  ultra thick] plot [smooth, tension=1] coordinates { (2.3,6.1) (1.63,4.85)};
\draw [black,  ultra thick] plot [smooth, tension=1] coordinates { (2.3,5.1) (1.7,3.85)};
\node [color=orange!65!white] at (2,7.45) {\tiny$\bullet$};
\node [color=orange!65!white] at (2,4.45) {\tiny$\bullet$};
\node [color=orange!65!white] at (2,5.54) {\tiny$\bullet$};
\node [color=orange!65!white] at (2,6.35) {\tiny$\bullet$};
\node [color=orange!65!white] at (2,-1.4) {$\bullet$};

\node [color=black]  at  (4.6,-2.8) {\tiny$ R''_1$};
\draw [black,  ultra thick] plot [smooth, tension=1] coordinates { (3.8,-.6) (4.3,-2.9)};
\draw [black,  ultra thick] plot [smooth, tension=1] coordinates { (3.7,8.5) (4.2,7.3)};
\draw [black,  ultra thick] plot [smooth, tension=1] coordinates { (3.8,7.5) (4.2,6.3)};
\draw [black,  ultra thick] plot [smooth, tension=1] coordinates { (3.8,4.2) (4.2,3.3)};
\draw [black,  ultra thick] plot [smooth, tension=1] coordinates { (3.8,3.3) (4.2,2.3)};
\node [color=blue!80!white] at (4,7.75) {\tiny$\bullet$};
\node [color=blue!80!white] at (4,6.87) {\tiny$\bullet$};
\node [color=blue!80!white] at (4,3.75) {\tiny$\bullet$};
\node [color=blue!80!white] at (4,2.85) {\tiny$\bullet$};
\node [color=blue!80!white] at (4.05,-1.7) {$\bullet$};

\node [color=black] at (0,-1.1) {$\bullet$};
\draw [thick,->] (0,1.5) -- (0,-.25);
\end{tikzpicture}
\caption{Admissible cover of degree 6}
\end{figure}
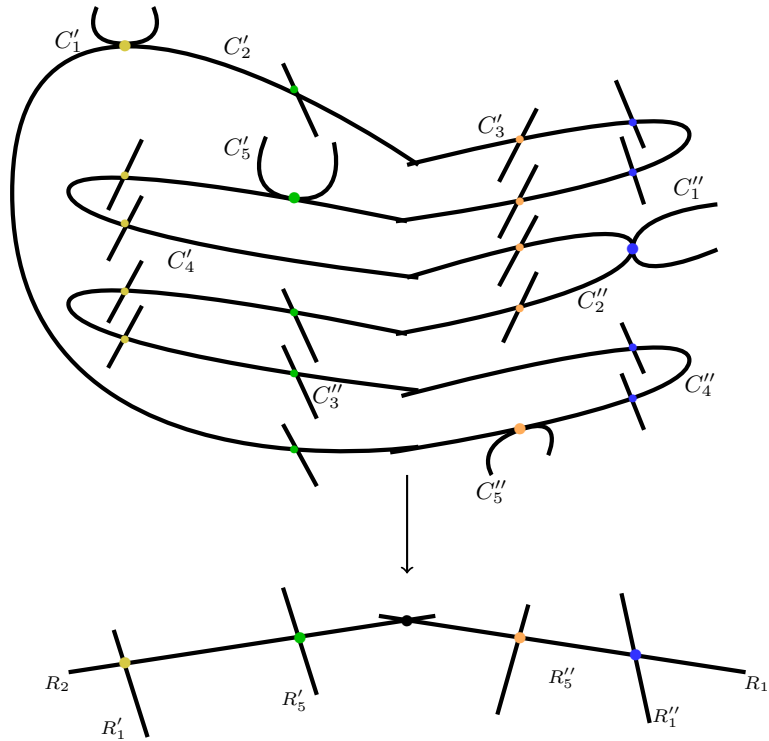

    \begin{rem}
     
     It is also easy to obtain an unramified double cover of the Brill-Noether curve to a curve of genus 5 (in lieu of the plane quintic): one should consider the curve obtained by attaching $C_2, C_3, C_4$ to each other at the points $Y_i, Z_i$ forming a circuit. 
      Attach then $C_1$ to $C_2$ and $C_5$ to $C_4$ using the point $X$. Then each $C'_i, C''_i$ can be mapped bijectively to the corresponding $C_i$
      giving rise to the double cover.
    \end{rem}
 
\begin{subsection}*{Acknowledgements} The second author would like to  express his gratitude to the Max-Planck Institut f\"ur Mathematik in Bonn,
 for the wonderful working conditions and stimulating environment. All three authors would like to thank Gavril Farkas for referring them to their work \cite{AF}.
\end{subsection}


\begin{thebibliography}{ACGH85}

\bibitem[AF12]{AF}
M.~Aprodu, G.~Farkas.
\newblock {\em Green's conjecture for general covers.}  Contemp. Math., 564, \newblock {\em Compact moduli spaces and vector bundles},
 211-226, Amer. Math. Soc., Providence, RI, 2012.

\bibitem[ACGH85]{ACGH}
E.~Arbarello, M.~Cornalba, P.~A.~ Griffiths, and J.~Harris.
\newblock {\em Geometry of algebraic curves. {V}ol. {I}}, volume 267 of {\em
  Grundlehren der Mathematischen Wissenschaften [Fundamental Principles of
  Mathematical Sciences]}.
\newblock Springer-Verlag, New York, 1985.

\bibitem[CLT]{CLT}
A.~Castorena, A.~L\'opez Mart\'in, and M.~ Teixidor i Bigas.
\newblock {\em Petri map for vector bundles near good bundles.} \href{http://arxiv.org/abs/1203.0983}{arXiv:1203.0983}

\bibitem[CDPR12]{CDPR}
F.~ Cools, J. ~Draisma, S.~ Payne, E. ~Robeva.
\newblock {\em A tropical proof of the Brill-Noether Theorem} 
\newblock{ Adv. Math. 230, 759--776, 2012.}

\bibitem[EH86]{EH86}
D.~Eisenbud and J.~Harris.
\newblock {\em Limit linear series: basic theory.}
\newblock { Invent. Math.}, 85(2):337--371, 1986.

\bibitem[EH87]{EH87}
D.~Eisenbud and J.~Harris.
\newblock {\em The {K}odaira dimension of the moduli space of curves of genus {$\geq
  23$}.}
\newblock { Invent. Math.}, 90(2):359--387, 1987.

\bibitem[Far10]{Fa10}
G.~Farkas.
\newblock {\em Rational maps between moduli spaces of curves and {G}ieseker-{P}etri
  divisors.}
\newblock {J. Algebraic Geom.}, 19(2):243--284, 2010.

\bibitem[HM82]{HM82}
J.~Harris and D.~Mumford.
\newblock {\em On the Kodaira dimension of the moduli space of curves.}
\newblock { Invent. Math.}, 67:23-86, 1982.

\bibitem[HM98]{HM98}
J.~Harris and I.~Morrison.
\newblock Moduli of curves.
\newblock {\em Graduate Texts in Mathematics}, 187. Springer-Verlag.

\bibitem[HT84]{HT84}
J.~Harris and L.~Tu.
\newblock {\em Chern numbers of kernel and cokernel bundles.}
\newblock { Invent. Math.}, 75(3):467--475, 1984.

\bibitem[Ort13]{Ort13}
A.~Ortega.
\newblock {\em The Brill-{N}oether curve and Prym-{T}yurin varieties.}
\newblock {Math. Annalen}, 3:809-817, 2013.

\bibitem[Oss05]{Oss05}
B.~Osserman.
\newblock {\em Two degeneration techniques for maps of curves.}  Contemp. Math., 388, 
\newblock {\em Snowbird lectures in algebraic geometry},
 137-143, Amer. Math. Soc., Providence, RI, 2005.
 
\bibitem[Oss]{Oss}
B.~Osserman.
\newblock  {\em Limit linear series moduli stacks in higher rank.}
\href{http://arxiv.org/abs/1405.2937}{arXiv:1405.2937}

\bibitem[Pir85]{Pir85}
G.~P.~Pirola.
\newblock {\em Chern character of degeneracy loci and curves of special divisors.}
\newblock { Ann. Mat. Pura Appl. (4)}, 142:77--90 (1986), 1985.

\bibitem[Sta99]{Sta99}
R.~P.~Stanley.
\newblock {\em Enumerative combinatorics. {V}ol. 2}, volume~62 of {\em
  Cambridge Studies in Advanced Mathematics}.
\newblock Cambridge University Press, Cambridge, 1999.
\newblock With a foreword by Gian-Carlo Rota and appendix 1 by Sergey Fomin.


\bibitem[Tei84] {T}  
M.~Teixidor i Bigas,
\newblock {\em For which Jacobi varieties is Sing $\Theta$ reducible?}
\newblock {J. Reine Angew. Math.} 354 (1984), 141-149

\bibitem[Tei14] {Toh}  
M.~Teixidor i Bigas,
\newblock {\em Limit linear series for vector bundles. }
\newblock {To appear in Tohoku Math. Journal.} 

\bibitem[Wel85]{W}
G.~E.~Welters.
\newblock {\em  A theorem of Gieseker-Petri type for Prym varieties.}
\newblock {Ann. Sci. \'Ecole Norm. Sup.} (4) 18: no. 4, 671--683, 1985.

\bibitem[Zei83]{Z}
D.~Zeilberger.
\newblock {\em Andr\'e's reflection proof generalized to the many-candidate ballot problem.}
\newblock { Discrete Math.} 44: 325--326, 1983.



\end{thebibliography}
\def\cprime{$'$}

\end{document}